\documentclass[b5paper,reqno]{amsart}
\usepackage{amsmath,amssymb,amsfonts,amsthm,latexsym,amstext,amscd}
\usepackage{graphicx}
\usepackage{enumerate}
\usepackage{hyperref}
\usepackage[left=1.5cm,top=3cm,right=1.5cm,bottom=3cm]{geometry}
\makeatletter \thm@headfont{\bfseries\scshape} \makeatother
\allowdisplaybreaks
\newtheorem{thm}{Theorem}[section]
\newtheorem{cor}[thm]{Corollary}
\newtheorem{lem}[thm]{Lemma}
\newtheorem{prop}[thm]{Proposition}
\theoremstyle{definition}

\theoremstyle{remark}
\newtheorem{rem}[thm]{Remark}
\numberwithin{equation}{section}

\begin{document}

\title[Finiteness of Tate-Shafarevich Group]{Finiteness of the Tate-Shafarevich Groups for Elliptic Curves over the Field of Rational Numbers}%
\author{Lan Nguyen }%
\address{Department of Mathematics}%
\email{nguyenl@uwp.edu}%

\thanks{}%

\begin{abstract}
Let $E$ be an elliptic curve over $\mathbb{Q}$. Let $\underline{III}(E)$ be a certain group of equivalence classes of homogeneous spaces of $E$ called its Tate-Shafarevich group. We show in this paper that this group has finite cardinality and discuss its role in the Birch and Swinnerton-Dyer Conjecture. In particular, our result implies the Parity Conjecture, or the Birch and Swinnerton-Dyer Conjecture modulo 2. It also removes the finiteness condition of $\underline{III}(E)$ from previous results in the literature of this subject and makes possible for some computation problems concerning the Strong Birch and Swinnerton-Dyer Conjecture. In addition, we also give an analogue of the Hasse-Minkowski Theorem for cubic plane curves.
\end{abstract}
\maketitle
\section{Introduction}
\subsection{Elliptic Curves and Tate-Shafarevich Groups} Let $E$ be an elliptic curve over $\mathbb{Q}$. Let $G_{\mathbb{Q}}$ denote the absolute Galois group Gal$(\overline{\mathbb{Q}}/\mathbb{Q})$. For each rational prime $p$, let $G_{\mathbb{Q}_{p}}$ denote the absolute Galois group Gal$(\overline{\mathbb{Q}_{p}}/\mathbb{Q}_{p})$. Let $WC(E_{/K}): = H^{1}(G_{K}, E)$ be the associated Weil-Ch\^{a}telet group where $K$ denotes either $\mathbb{Q}$ or $\mathbb{Q}_{p}$ for some rational prime $p$, finite or infinite. The Tate-Shafarevich group $\underline{III}(E)$ of $E$ is defined by Lang-Tate ([12]) and Shafarevich ([21]) in this case as the collection of elements of $WC(E_{/\mathbb{Q}})$ that becomes trivial in all completions of $\mathbb{Q}$. In other words, $\underline{III}(E)$ consists of all homogeneous spaces of $E$, up to equivalence, which have points everywhere locally. In the language of Galois Cohomology, $\underline{III}(E)$ can be written as \begin{equation}\underline{III}(E): = \bigcap_{p}ker(H^{1}(G_{\mathbb{Q}}, E)\rightarrow H^{1}(G_{\mathbb{Q}_{p}}, E)).\end{equation}
The group $\underline{III}(E)$ has been conjectured by Tate and Shafarevich ([12, 21, 23]), called the Tate Shafarevich Conjecture in the literature, to be finite. Moreover, Cassels ([3]) and Tate show that if the cardinality of $\underline{III}(E)$ is finite, then it is a perfect square due to the existence of a certain alternate bilinear pairing on it (see [16] for polarized abelian variety case). The group $\underline{III}(E)$ plays a vital role in the Birch and Swinnerton-Dyer Conjecture (BSDC), particularly the strong Birch and Swinnerton Conjecture ([1]). In fact, Artin and Tate ([22]) prove that, in the function field case, the finiteness of $\underline{III}(E)$ is equivalent to the entire BSDC. To state this conjecture, we need to define some terminology.

Let $E$ be an elliptic curve over $\mathbb{Q}$. Let $E(\mathbb{Q})$ denote the set of points of $E$ defined over $\mathbb{Q}$. By Mordell's Theorem ([15]), $E(\mathbb{Q})$ is finitely generated and
$$E(\mathbb{Q}) \cong \mathbb{Z}^{r} \times E(\mathbb{Q})_{tor}$$ where $E(\mathbb{Q})_{tor}$ denotes the collection of rational points of finite order, a finite group, and $r$, a nonnegative integer, is called the Weil-Mordell rank of $E$ and denoted by $r_{WM}(E)$. The group $E(\mathbb{Q})_{tor}$ is completely understood due to a result of Mazur ([14]). The $L$-function of $E$ in a complex variable $s$ is defined as the Euler product $$L_{E}(s): = \prod_{p \nmid N_{E}}(1-a_{p}p^{-s}+p^{1-2s})^{-1}\prod_{p|N_{E}}(1-a_{p}p^{-s})^{-1}$$ where $N_{E}$ denotes the conductor of $E$. Define $$r_{anal}(E): = ord_{s=1}L_{E}(s),$$ a nonnegative integer.

\bigskip

\textbf{Strong Birch and Swinnerton-Dyer Conjecture}: Let $E$ be an elliptic curve over $\mathbb{Q}$. Then:
\begin{enumerate}
\item $r_{anal}(E) = r_{MW}(E)$ where $r_{anal}$ and $r_{WM}$ are the analytic and the Weil-Mordell ranks of $E$ respectively.
\item $\underline{III}(E)$ is finite. 
\item \begin{equation}c_{E} = \lim_{s \rightarrow 1}\frac{L_{E}(s)}{(s-1)^{r_{MW}}} = \frac{|\underline{III}(E)|\Omega(E)Reg(E)(\prod_{p}c_{p})}{|E(\mathbb{Q})_{tor}^{2}|}\end{equation} where:
    \begin{itemize}
    \item $c_{E}$ is the coefficients of the term $(s-1)^{r_{anal}}$ in the Taylor expansion of $L_{E}(s)$ at $s = 1$.
    \item $\Omega(E)$ is the real period of $E$ multiplied by the number components of $E(\mathbb{R})$.
    \item $c_{p} = [E(\mathbb{Q}_{p}): E^{0}(\mathbb{Q}_{p})]$ for each prime $p$ is the corresponding Tamagawa number.
    \item $Reg(E)$ is the regulator of $E$, which is the determinant of the height pairing.
    \item $E(\mathbb{Q})_{tor}$ is the group of rational torsion points of $E$.
    \end{itemize}
\end{enumerate}

One of the first partial result, toward proving the BSDC, was obtained by J. Coates and A. Wiles ([5]). Now the full conjecture is known for all elliptic curves over $\mathbb{Q}$ with rank at most one.

\bigskip

The question of the finiteness of $\underline{III}(E)$ presents a major obstacle in answering many fundamental questions, both theoretically and computationally, toward understanding the Birch and Swinnerton-Dyer Conjecture. Theoretically, one of such important questions is the Parity Conjecture ([8]). Computationally,  important examples of such questions are the verification of $c_{E}$ in the strong Birch and Swinnerton-Dyer Conjecture and the determination of existence of rational points of an arbitrary genus one curve over $\mathbb{Q}$. Similar situations can be seen throughout the literature on this subject. Even in situations where it is possible to deduce interesting arguments to circumvent this obstacle, it often complicates and lengthens the arguments (see [4] for a discussion of this).

\subsection{Hasse Principle}

One of the fundamental questions in Diophantine Number Theory is whether a rational Diophantine equation or a system of such equations has a solution in $\mathbb{Q}$. Answering such a question is difficult and not always possible. Such questions are addressed by Hilbert's 10th problem ([13]).

If a polynomial with rational coefficients has a nontrivial rational solution, then it has nontrivial solutions in all $p$-adic fields $\mathbb{Q}_{p}$ and in $\mathbb{R}$. The Hasse principle asks when the reverse direction is true. That is, if a polynomial with rational coefficients has a nontrivial solution in each $p$-adic field $\mathbb{Q}_{p}$ and in $\mathbb{R}$, does it have a nontrivial rational solution? If it does, then one says that it satisfies the Hasse principle. More generally, the Hasse principle for a variety is a statement about the existence of global points given the existence of local points. For an elliptic curve $E$ over $\mathbb{Q}$, Manin shows that the obstruction for the Hasse principle is completely accounted for by $\underline{III}(E)$. That is, if $\underline{III}(E)$ is nontrivial, then there is at least one homogenous space of $E$ which has points over all local fields but does not have a global point. For a quadratic form, Minkowski established around 1920 (generalized to arbitrary number fields by Hasse later) the following beautiful theorem ([20]):

\begin{thm}{(Hasse-Minkowski Theorem)}
Let $F$ be a quadratic form with rational coefficients. Then $F$ has a nontrivial rational solution if and only if $F$ has a nontrivial solution in each completion $\mathbb{Q}_{p}$ of $\mathbb{Q}$ where $p$ ranges over all finite and infinite primes.
\end{thm}

For cubic forms, Selmer shows that the Hasse principle fails with the following well-known counterexample: \begin{equation}3x^{3}+4y^{3}+5z^{3} = 0.\end{equation} H. Davenport ([7]) shows that the Hasse principle holds trivially for cubic forms with rational coefficients in at least 16 variables.  Roger Heath-Brown proves a similar result, using the Hardy-Littlewood circle method, for cubic forms with rational coefficients in at least 14 variables ([9]). For nonsingular cubic forms with rational coefficients, Hooley ([10]) uses a similar method to show that the Hasse principle holds for forms with 9 or more variables.

\section{Main Results}
The problems considered in this paper belong to the general area of local-global principle. Among the results in this area, the Hasse-Minkowski Theorem mentioned above, which treats the case of quadratic forms, is one of the most well-known results. The main results of this paper are our resolutions of the Shafarevich and Tate Conjecture (which has been conjectured by Shafarevich, Tate, Cassels, Birch and Swinnerton-Dyer around 1958-1959 ([12, 21, 23])) which is part of the Strong Birch and Swinnerton-Dyer Conjecture ([1]), and the Parity Conjecture ([8]). These results also make it possible for some important computational problems concerning the strong Birch and Swinnerton-Dyer Conjecture to be carried out.
\begin{thm}
Let $E$ be an elliptic curve over $\mathbb{Q}$. Then the associated Tate-Shafarevich group $\underline{III}(E)$ has finite cardinality.
\end{thm}
\begin{rem}
Theorem 2.1 is proved by Rubin ([18]) for elliptic curves $E$ having complex multiplication and Weil-Mordell rank at most 1. It is then extended by Kolyvagin ([11]) to modular elliptic curves $E$ with Weil-Mordell rank at most 1 and then to all elliptic curves $E$ over $\mathbb{Q}$ of Weil-Mordell rank at most 1 due to the Modularity Theorem of Breuil, Conrad, Diamond, and Taylor ([2]).
\end{rem}
Since there are a number of results in the literature depending on the finiteness of $\underline{III}(E)$, we select a few important representatives as consequences of our main result.
\begin{cor}
The Parity Conjecture holds. In other words,  $$r_{anal}(E) \equiv r_{WM}(E) \pmod{2}$$ where $r_{anal}(E)$ and $r_{WM}(E)$ are the analytic and the Weil-Mordell ranks of $E$ respectively.
\end{cor}
For computational problems concerning the strong Birch and Swinnerton-Dyer Conjecture,  we have the following results ([6, 23]):
\begin{cor}
Let $E$ be an elliptic curve over $\mathbb{Q}$. Then the standard process always produces generators for the group $E(\mathbb{Q})$ and thus the term $c_{E}$ in the Strong Birch and Swinnerton-Dyer Conjecture is computable.
\end{cor}
\begin{cor}
There is an algorithm with a finite number of steps which determines whether a genus one curve $\mathcal{C}$ with rational coefficients has a nontrivial rational solution.
\end{cor}

\section{Proof of Results}
\begin{proof}{(Proof of Theorem 2.1)}
To prove Theorem 2.1, we first prove the following lemma:

\begin{lem}{(Key Lemma)}
Let $F(x, y, z)$ be a cubic plane curve with rational coefficients, i.e., $F(x, y, z)$ is a cubic form
$$A_{1}x^{3}+A_{2}y^{3}+A_{3}z^{3}+A_{4}x^{2}y+A_{5}x^{2}z+A_{6}y^{2}x+A_{7}y^{2}z+A_{8}z^{2}x+A_{9}z^{2}y+A_{10}xyz$$ where $A_{i}$ is a rational number for each $i$. Suppose that $F(x, y, z) = 0$ has a nontrivial solution in each $\mathbb{Q}_{p}^{3}$ where $p$ ranges over all primes, finite and infinite. If $F(x, y, z)$ does not have a nontrivial rational solution then $F(x, y, z)$ is of the form $$F(x, y, z) = A_{1}x^{3} + A_{2}y^{3} + A_{3}z^{3}$$ such that $A_{1}A_{2}A_{3} \neq 0$.
\end{lem}

\begin{rem}
The Key Lemma above can be viewed as an analogue of the Hasse-Minkowski Theorem for cubic plane curves.
\end{rem}

\begin{proof}{(Proof of Key Lemma)}

In the rest of the paper, a local or global solution \begin{equation}(x_{0}, y_{0}, z_{0}) \end{equation} of $F(x, y, z) = 0$ is said to be nontrivial if at least one of the components in nonzero.

Let $F(x, y, z) = 0$ be a cubic plane curve with rational coefficients. Thus it can be written in the form:
\begin{equation}A_{1}x^{3}+A_{2}y^{3}+A_{3}z^{3}+A_{4}x^{2}y+A_{5}x^{2}z+A_{6}y^{2}x+A_{7}y^{2}z+A_{8}z^{2}x+A_{9}z^{2}y+A_{10}xyz = 0\end{equation} with rational coefficients $A_{i}$ for each $i$. Suppose that the cubic form $F$ has nontrivial local solutions at all places. We want to study whether $F$ has a nontrivial rational solution. We may assume that \begin{equation}A_{1}A_{2}A_{3} \neq 0\end{equation} since otherwise, say $A_{1} = 0$, then $(1, 0, 0)$ is a nontrivial rational solution of $F$. Also, it follows from (3.3) that if $(u, v, t)$ is a nontrivial solution of $F$, then at most one of the components can be zero.

For the rest of the proof of Theorem 2.1, let us suppose that at least one of the coefficients $A_{i}$ for $i > 3$ is nonzero. By multiply both sides of (3.2) by $x$, $y$, and $z$, we obtain respectively:
\begin{equation}
A_{1}x^{4}+A_{2}xy^{3}+A_{3}xz^{3}+A_{4}x^{3}y+A_{5}x^{3}z+A_{6}y^{2}x^{2}+A_{7}y^{2}zx+A_{8}z^{2}x^{2}+A_{9}z^{2}yx+A_{10}x^{2}yz  = 0,
\end{equation}
\begin{equation}
A_{1}x^{3}y+A_{2}y^{4}+A_{3}z^{3}y+A_{4}x^{2}y^{2}+A_{5}x^{2}zy+A_{6}y^{3}x+A_{7}y^{3}z+A_{8}z^{2}xy+A_{9}z^{2}y^{2}+A_{10}xy^{2}z  = 0,
\end{equation}
\begin{equation}
A_{1}x^{3}z+A_{2}y^{3}z+A_{3}z^{4}+A_{4}x^{2}yz+A_{5}x^{2}z^{2}+A_{6}y^{2}xz+A_{7}y^{2}z^{2}+A_{8}z^{3}x+A_{9}z^{3}y+A_{10}xyz^{2}  = 0.
\end{equation}

Let $X: = x^{2}$, $Y: = y^{2}$, $Z: = z^{2}$, $W: = xy$, $M: = xz$, and $N: = yz$. Then
\begin{equation}WM - XN = 0,\end{equation}
\begin{equation}MN - ZW = 0, \end{equation}
\begin{equation}WN - YM = 0, \end{equation}
\begin{equation}W^{2} - XY = 0, \end{equation}
\begin{equation}M^{2} - XZ = 0,\end{equation}
\begin{equation}N^{2} - YZ = 0. \end{equation}
Thus (3.4), (3.5), and (3.6) can be rewritten respectively as
\begin{equation}
A_{1}X^{2}+A_{2}YW+A_{3}ZM+A_{4}XW+A_{5}XM+A_{6}XY+A_{7}YM+A_{8}XZ+A_{9}ZW+A_{10}XN = 0,
\end{equation}
\begin{equation}
A_{1}XW+A_{2}Y^{2}+A_{3}ZN+A_{4}XY+A_{5}XN+A_{6}YW+A_{7}YN+A_{8}ZW+A_{9}YZ+A_{10}YM = 0,
\end{equation}
\begin{equation}
A_{1}XM+A_{2}YN+A_{3}Z^{2}+A_{4}XN+A_{5}XZ+A_{6}YM+A_{7}YZ+A_{8}ZM+A_{9}ZN+A_{10}ZW = 0.
\end{equation}

Note that quadratic forms (3.7)-(3.15) have common nontrivial local solutions at all places since $F(x, y, z) = 0$ is assumed to have nontrivial local solutions at all places. For convenience, we denote (3.13), (3.14), and (3.15) by $F_{x}(X, Y, Z, W, M, N) = 0$, $F_{y}(X, Y, Z, W, M, N) = 0$, and $F_{z}(X, Y, Z, W, M, N) = 0$ respectively where each index indicates which variable is multiplied to both sides of (3.2) to form the corresponding equation.

Since we assume that $A_{1}A_{2}A_{3} \neq 0$, it can be seen from equations (3.13)-(3.15) that two of the three variables $W$, $M$, and  $N$ appear in each of these equations with nonzero coefficients. Since we also assume that $A_{i} \neq 0$ for at least one $i > 3$, it can be verified that at least one of the equations (3.13)-(3.15) can be rewritten, using equations (3.7)-(3.9), so that all three variables $W$, $M$, and  $N$ appear in it with nonzero coefficients. Note that in these rewritten equations, the terms with coefficients $A_{1}$, $A_{2}$, and $A_{3}$ are unchanged.

\begin{prop}

The system of equations \begin{equation}\begin{cases} F_{x}(X, Y, Z, W, M, N) = 0 \\ F_{y}(X, Y, Z, W, M, N) = 0 \\ F_{z}(X, Y, Z, W, M, N) = 0 \\ WM - XN = 0 \\ MN - ZW = 0 \\ WN - YM = 0 \\ W^{2} - XY = 0 \\ M^{2} - XZ = 0 \\ N^{2} - YZ = 0 \end{cases}\end{equation} has a common nontrivial rational solution if and only if equation (3.2) has a nontrivial rational solution.
\end{prop}

\begin{proof}
If equation (3.2) has a nontrivial rational solution, then it is clear from its construction that the system of equations (3.16) has a common nontrivial rational solution. Suppose that the system of equations (3.16) has a common nontrivial rational solution, say \begin{equation}(X_{0}, Y_{0}, Z_{0}, W_{0}, M_{0}, N_{0}) \neq (0, 0, 0, 0, 0, 0).\end{equation}

It follows from (3.3) and (3.13)-(3.15) that at least two of the three components $X_{0}, Y_{0}$, and $Z_{0}$ must be nonzero. If all three components are nonzero, then it can be verified that all the rest of the components must also be nonzero. If one of the components $X_{0}, Y_{0}$, and $Z_{0}$ is zero, then two of the components $W_{0}$, $M_{0}$, and $N_{0}$ must be zero. For example, if $X_{0} = 0$, then it follows that $W_{0} = 0$ and $M_{0} = 0$. Now let us suppose that $X_{0} \neq 0$. Then \begin{equation}\bigg(1, \frac{Y_{0}}{X_{0}},\frac{Z_{0}}{X_{0}}, \frac{W_{0}}{X_{0}}, \frac{M_{0}}{X_{0}}, \frac{N_{0}}{X_{0}}\bigg)\end{equation} is also a nontrivial common rational solution of system (3.16). Also, \begin{equation}\begin{cases} W_{0}^{2} - X_{0}Y_{0} = 0 \\ M_{0}^{2} - X_{0}Z_{0} = 0 \\ N_{0}^{2} - Y_{0}Z_{0} = 0 \end{cases} \end{equation} implies that \begin{equation}\begin{cases} \bigg(\frac{W_{0}}{X_{0}}\bigg)^{2} = \frac{Y_{0}}{X_{0}} \\ \bigg(\frac{M_{0}}{X_{0}}\bigg)^{2} = \frac{Z_{0}}{X_{0}} \\ \bigg(\frac{N_{0}}{X_{0}}\bigg)^{2} = \frac{Y_{0}}{X_{0}}\frac{Z_{0}}{X_{0}}. \end{cases} \end{equation} It is then straight forward, using (3.2), (3.7), and (3.13), to verify that \begin{equation}\bigg(1, \frac{W_{0}}{X_{0}}, \frac{M_{0}}{X_{0}}\bigg) \end{equation} is a nontrivial rational solution of (3.2). Similar arguments work for $Y_{0} \neq 0$ and for $Z_{0} \neq 0$. Note that $A_{1}A_{2}A_{3} \neq 0$ implies that if $(X_{0}, Y_{0}, Z_{0}, W_{0}, M_{0}, N_{0})$ is a nontrivial solution of system (3.16), then at most one of the numbers $X_{0}$, $Y_{0}$, $Z_{0}$ can be zero.

\end{proof}

\begin{rem}
As a result of the proof of Proposition 3.3, it can be verified from construction of $(3.13)-(3.15)$ that (3.2) has a nontrivial rational solution $(x_{0}, y_{0}, z_{0})$ with

\begin{itemize}

\item $x_{0} \neq 0 $ if and only if system (3.16), without equations $F_{y}(X, Y, Z, W, M, N)$ and/or $F_{z}(X, Y, Z, W, M, N)$, has a nontrivial solution $(X_{0}, Y_{0}, Z_{0}, W_{0}, M_{0}, N_{0})$ with $X_{0} \neq 0$;
\item $y_{0} \neq 0 $ if and only if system (3.16), without equations $F_{x}(X, Y, Z, W, M, N)$ and/or $F_{z}(X, Y, Z, W, M, N)$, has a nontrivial solution $(X_{0}, Y_{0}, Z_{0}, W_{0}, M_{0}, N_{0})$ with $Y_{0} \neq 0$;
\item $z_{0} \neq 0 $ if and only if system (3.16), without equations $F_{x}(X, Y, Z, W, M, N)$ and/or equation $F_{y}(X, Y, Z, W, M, N)$, has a nontrivial solution $(X_{0}, Y_{0}, Z_{0}, W_{0}, M_{0}, N_{0})$ with $Z_{0} \neq 0$.
\end{itemize}
\end{rem}

\begin{prop}

System of equations (3.16) has a common nontrivial rational solution $$(X_{0}, Y_{0}, Z_{0}, W_{0}, M_{0}, N_{0})$$ with $$X_{0}Y_{0}Z_{0} \neq 0$$ if and only if system of equations (3.16) without equation (3.10) has a common nontrivial rational solution $(X_{0}, Y_{0}, Z_{0}, W_{0}, M_{0}, N_{0})$ with $X_{0}Y_{0}Z_{0} \neq 0$. The same statement holds when (3.11) or (3.12) replaces (3.10).
\end{prop}

\begin{proof}
If system (3.16) has a common nontrivial rational solution $(X_{0}, Y_{0}, Z_{0}, W_{0}, M_{0}, N_{0})$ with $X_{0}Y_{0}Z_{0} \neq 0$ (and thus $W_{0}M_{0}N_{0} \neq 0$), then it is clear that system (3.16) without equation (3.10) has a common nontrivial rational solution $(X_{0}, Y_{0}, Z_{0}, W_{0}, M_{0}, N_{0})$ with $X_{0}Y_{0}Z_{0} \neq 0$. Now suppose system (3.16) without equation (3.10) has a common nontrivial rational solution $(X_{0}, Y_{0}, Z_{0}, W_{0}, M_{0}, N_{0})$ with $X_{0}Y_{0}Z_{0} \neq 0$. Since $Y_{0} \neq 0$ and \begin{equation}W_{0}N_{0} - Y_{0}M_{0} = 0, \end{equation} \begin{equation}M_{0} = \frac{W_{0}N_{0}}{Y_{0}}.\end{equation} Since $X_{0}Y_{0}Z_{0} \neq 0$ implies $N_{0} \neq 0$, \begin{equation}W_{0}M_{0} - X_{0}N_{0} = 0,\end{equation} together with (3.23) give \begin{equation}\frac{N_{0}W_{0}^{2}}{Y_{0}} = X_{0}N_{0}\end{equation} and thus \begin{equation}W^{2}_{0} = X_{0}Y_{0}.\end{equation} Therefore, $(X_{0}, Y_{0}, Z_{0}, W_{0}, M_{0}, N_{0})$ is also a nontrivial rational solution to equation (3.10) and thus a common nontrivial rational solution to system (3.16). Similar arguments work when (3.11) or (3.12) replaces (3.10).

\end{proof}

\begin{rem}
From \textbf{(3.2)}, we have
\begin{equation}A_{3}z^{3}+A_{5}x^{2}z+A_{7}y^{2}z+A_{8}z^{2}x+A_{9}z^{2}y+A_{10}xyz = A_{1}x^{3}+A_{2}y^{3}+A_{4}x^{2}y+A_{6}y^{2}x,\end{equation}
\begin{equation}A_{2}y^{3}+A_{4}x^{2}y+A_{6}y^{2}x+A_{7}y^{2}z+A_{9}z^{2}y+A_{10}xyz = A_{1}x^{3}+A_{3}z^{3}+A_{5}x^{2}z+A_{8}z^{2}x,\end{equation}
\begin{equation}A_{1}x^{3}+A_{4}x^{2}y+A_{5}x^{2}z+A_{6}y^{2}x+A_{8}z^{2}x+A_{10}xyz = A_{2}y^{3}+A_{3}z^{3}+A_{7}y^{2}z+A_{9}z^{2}y.\end{equation}

If \begin{equation}A_{1}x^{3}+A_{2}y^{3}+A_{4}x^{2}y+A_{6}y^{2}x = 0\end{equation} has a nontrivial rational solution $(x_{0}, y_{0})$, then $(x_{0}, y_{0}, 0)$ is a nontrivial rational solution of (3.2) by (3.27). Thus let us assume that is not the case. Similarly, we assume that \begin{equation}A_{1}x^{3}+A_{3}z^{3}+A_{5}x^{2}z+A_{8}z^{2}x = 0\end{equation} and \begin{equation}A_{2}y^{3}+A_{3}z^{3}+A_{7}y^{2}z+A_{9}z^{2}y = 0\end{equation} do not have nontrivial rational solutions.
\end{rem}

Next we need a result involving the determinant of certain matrix
\begin{prop}
Let $\mathcal{A}_{1}$, $\mathcal{B}_{1}$, $\mathcal{C}_{1}$, $\mathcal{A}_{2}$, $\mathcal{B}_{2}$, and $\mathcal{C}_{2}$ be six elements of a ring $\mathcal{R}$. Then
\begin{equation}Det\left(
  \begin{array}{cccc}
    \mathcal{A}_{1} & \mathcal{B}_{1} & \mathcal{C}_{1} & 0 \\
    0 & \mathcal{A}_{1} & \mathcal{B}_{1} & \mathcal{C}_{1}  \\
    \mathcal{A}_{2} & \mathcal{B}_{2} & \mathcal{C}_{2}  & 0 \\
    0 & \mathcal{A}_{2} & \mathcal{B}_{2} & \mathcal{C}_{2}  \\
  \end{array}
\right) = (\mathcal{A}_{1}\mathcal{C}_{2} - \mathcal{A}_{2}\mathcal{C}_{1})^{2} + (\mathcal{A}_{2}\mathcal{B}_{1} - \mathcal{A}_{1}\mathcal{B}_{2})(\mathcal{B}_{1}\mathcal{C}_{2} - \mathcal{B}_{2}\mathcal{C}_{1}).
\end{equation}
\end{prop}

\begin{proof}
It follows by a direct computation.

\end{proof}

\begin{prop}{(Key Proposition)}

The system of equations \begin{equation}\begin{cases} F_{x}(X, Y, Z, W, M, N) = 0 \\ F_{y}(X, Y, Z, W, M, N) = 0 \\ WM - XN = 0 \\ MN - ZW = 0 \\ WN - YM = 0 \\ M^{2} - XZ = 0 \\ N^{2} - YZ = 0 \end{cases}\end{equation} has a common nontrivial rational solution $(X_{0}, Y_{0}, Z_{0}, W_{0}, M_{0}, N_{0})$ with $$X_{0}Y_{0}Z_{0} \neq 0$$ which is also a common nontrivial rational solution of system (3.16).
\end{prop}

\begin{proof}
Since $F(x, y, z) = 0$ has nontrivial local solutions at all places, it can be verified that system (3.34) has common nontrivial local solutions at all places. First, we need the following lemma:

\begin{lem}
If $(X_{0}, Y_{0}, Z_{0}, W_{0}, M_{0}, N_{0})$ is a common nontrivial rational solution of system (3.34), then $X_{0}Y_{0} \neq 0$.
\end{lem}

\begin{proof}
If $(X_{0}, Y_{0}, Z_{0}, W_{0}, M_{0}, N_{0})$ is a common nontrivial rational solution of system (3.34), then it is a common nontrivial rational solution of systems \begin{equation}\begin{cases} F_{x}(X, Y, Z, W, M, N) = 0 \\ M^{2} - XZ = 0 \end{cases} \end{equation} and \begin{equation}\begin{cases} F_{y}(X, Y, Z, W, M, N) = 0 \\ N^{2} - YZ = 0. \end{cases} \end{equation} If $Y_{0} = 0$, then it follows from the assumption $A_{1}A_{2}A_{3} \neq0$ that $X_{0}Z_{0} \neq 0$. It also follows from $Y_{0} = 0$, via equation (3.13) and system (3.35), that \begin{equation}\begin{cases}A_{1}X^{2} + A_{3}ZM + A_{5}XM + A_{8}XZ = 0 \\ M^{2} - XZ = 0 \end{cases}\end{equation} has $(X_{0}, Z_{0}, M_{0})$ as a common nontrivial rational solution. Hence, $$(1, \frac{Z_{0}}{X_{0}}, \frac{M_{0}}{X_{0}})$$ is also a common nontrivial rational solution of (3.37), and thus equation $$M^{2} - XZ = 0$$ implies that  \begin{equation}\bigg(\frac{M_{0}}{X_{0}}\bigg)^{2} = \frac{Z_{0}}{X_{0}}.  \end{equation} Then it can be verified that \begin{equation}\bigg(1, \frac{M_{0}}{X_{0}}\bigg)\end{equation} is a nontrivial rational solution of $$A_{1}x^{3} + A_{3}z^{3} + A_{5}x^{2}z + A_{8}z^{2}x = 0,$$ which contradicts our assumptions in Remark 3.6.

If $X_{0} = 0$, then the assumption $A_{1}A_{2}A_{3} \neq0$ implies that $Y_{0}Z_{0} \neq 0$. Together with equation (3.14) and system (3.36), it can be verified that \begin{equation}\begin{cases}A_{2}Y^{2} + A_{3}ZN + A_{7}YN + A_{9}YZ = 0 \\ N^{2} - YZ = 0 \end{cases}\end{equation} has $(Y_{0}, Z_{0}, N_{0})$ as a common nontrivial rational solution. Since $Y_{0} \neq 0$, $$(1, \frac{Z_{0}}{Y_{0}}, \frac{N_{0}}{Y_{0}})$$ is also a common nontrivial rational solution of (3.40). Hence, equation $$N^{2} - YZ$$ implies that $$\bigg(\frac{N_{0}}{Y_{0}}\bigg)^{2} = \frac{Z_{0}}{Y_{0}}.$$ As a result, it can be verified that $$(1, \frac{N_{0}}{Y_{0}})$$ is a nontrivial rational solution of $$A_{2}y^{3}+A_{3}z^{3}+A_{7}y^{2}z+A_{9}z^{2}y = 0,$$ which again contradicts the assumption made in Remark 3.6.

\end{proof}

Define $$\mathfrak{F}_{1}: = \alpha_{1}(F_{x}(X, Y, Z, W, M, N)) + \alpha_{2}(F_{y}(X, Y, Z, W, M, N)) + \alpha_{3}(WM - XN) + \alpha_{4}(MN - ZW) +$$ $$+ \alpha_{5}(WN - YM) + \alpha_{6}(M^{2} - XZ) + \alpha_{7}(N^{2} - YZ)$$ and $$\mathfrak{F}_{2}: = \beta_{1}(F_{x}(X, Y, Z, W, M, N)) + \beta_{2}(F_{y}(X, Y, Z, W, M, N)) + \beta_{3}(WM - XN) + \beta_{4}(MN - ZW) +$$$$ + \beta_{5}(WN - YM) + \beta_{6}(M^{2} - XZ) + \beta_{7}(N^{2} - YZ)$$ where $\alpha_{i}$ and $\beta_{i}$ are indeterminates for each $i$. Let $$\mathfrak{R}: = \mathfrak{R}(X, Y, Z, M, N, \alpha_{1}, \alpha_{2}, \alpha_{3}, \alpha_{4}, \alpha_{5}, \alpha_{6}, \alpha_{7}, \beta_{1}, \beta_{2}, \beta_{3}, \beta_{4}, \beta_{5}, \beta_{6}, \beta_{7}) $$ be the Resultant of $\mathfrak{F}_{1}$ and $\mathfrak{F}_{2}$ with respect to the variable $W$. Then $\mathfrak{R}$ is the determinant of the Sylvester matric $$\mathfrak{S}: = \mathfrak{S}(X, Y, Z, M, N, \alpha_{1}, \alpha_{2}, \alpha_{3}, \alpha_{4}, \alpha_{5}, \alpha_{6}, \alpha_{7}, \beta_{1}, \beta_{2}, \beta_{3}, \beta_{4}, \beta_{5}, \beta_{6}, \beta_{7})$$ with $\mathfrak{S}$ being the following matrix \begin{equation}\left(
  \begin{array}{cccc}
    \mathcal{A}_{1} & \mathcal{B}_{1} & \mathcal{C}_{1} & 0 \\
    0 & \mathcal{A}_{1} & \mathcal{B}_{1} & \mathcal{C}_{1}  \\
    \mathcal{A}_{2} & \mathcal{B}_{2} & \mathcal{C}_{2}  & 0 \\
    0 & \mathcal{A}_{2} & \mathcal{B}_{2} & \mathcal{C}_{2}  \\
  \end{array}
\right)
\end{equation} where:

\begin{itemize}

\item $\mathcal{A}_{1}$ is the coefficient of $W^{2}$ of $\mathfrak{F}_{1}$.

\item $\mathcal{B}_{1}$ is the coefficient of $W^{1}$ of $\mathfrak{F}_{1}$.

\item $\mathcal{C}_{1}$ is the coefficient of $W^{0}$ of $\mathfrak{F}_{1}$.

\item $\mathcal{A}_{2}$ is the coefficient of $W^{2}$ of $\mathfrak{F}_{2}$.

\item $\mathcal{B}_{2}$ is the coefficient of $W^{1}$ of $\mathfrak{F}_{2}$.

\item $\mathcal{C}_{2}$ is the coefficient of $W^{0}$ of $\mathfrak{F}_{2}$.

\end{itemize} Since $\mathcal{A}_{1} = \mathcal{A}_{2} = 0$, it follows from Proposition 3.7 that \begin{equation}\mathfrak{R}: = \mathfrak{R}(X, Y, Z, M, N, \alpha_{1}, \alpha_{2}, \alpha_{3}, \alpha_{4}, \alpha_{5}, \alpha_{6}, \alpha_{7}, \beta_{1}, \beta_{2}, \beta_{3}, \beta_{4}, \beta_{5}, \beta_{6}, \beta_{7}) = 0.\end{equation} That is $\mathfrak{R}$, viewed as a polynomial in $X$, $Y$, $Z$, $W$, $M$, and $N$, is identically zero. Therefore, all the equations in system (3.34) have a common factor $$\mathfrak{C}(X, Y, Z, W, M, N),$$ a homogeneous polynomial with rational coefficients. There are two cases:
\begin{enumerate}
\item deg$(\mathfrak{C}(X, Y, Z, W, M, N)) = 1$.
\item deg$(\mathfrak{C}(X, Y, Z, W, M, N)) = 2$.
\end{enumerate}

If deg$(\mathfrak{C}(X, Y, Z, W, M, N)) = 1$, then system (3.34) must have a nontrivial common rational solution. Together with Lemma 3.9, it can be verified that \begin{equation}\mathfrak{C}(X, Y, Z, W, M, N) = aX + bY \end{equation} for some rational numbers $a$ and $b$ such that $ab \neq 0$. Let $Z_{1}$ be a nonzero rational number. It follows from (3.43) that there exists a nontrivial rational solution $(X_{0}, Y_{0}, Z_{0}, W_{0}, M_{0}, N_{0})$ such that $X_{0} = -b$, $Y_{0} = a$ and $Z_{0} = LZ_{1}$ for some nonzero rational number $L$. Thus $X_{0}Y_{0}Z_{0} \neq 0$ as required.

If deg$(\mathfrak{C}(X, Y, Z, W, M, N)) = 2$, i.e., a quadratic form, then it can be verified that there exist nonzero rational numbers $E_{1}$, $E_{2}$, $E_{3}$, $E_{4}$, and $E_{5}$ such that:
\begin{equation}
F_{x}(X, Y, Z, W, M, N) = E_{1}\mathfrak{C}(X, Y, Z, W, M, N),
\end{equation}
\begin{equation}
F_{y}(X, Y, Z, W, M, N) = E_{2}\mathfrak{C}(X, Y, Z, W, M, N),
\end{equation}
\begin{equation}
WM - XN = E_{3}\mathfrak{C}(X, Y, Z, W, M, N),
\end{equation}
\begin{equation}
MN - ZW = E_{4}\mathfrak{C}(X, Y, Z, W, M, N),
\end{equation}
\begin{equation}
WN - YM = E_{5}\mathfrak{C}(X, Y, Z, W, M, N),
\end{equation}
\begin{equation}
M^{2} - XZ = E_{6}\mathfrak{C}(X, Y, Z, W, M, N),
\end{equation}
\begin{equation}
N^{2} - YZ = E_{7}\mathfrak{C}(X, Y, Z, W, M, N).
\end{equation} Since system (3.34) has common nontrivial local solutions at all places, $\mathfrak{C}(X, Y, Z, W, M, N)$ also does. Then by the Hasse-Minkowski Theorem for a quadratic form, $\mathfrak{C}(X, Y, Z, W, M, N)$ has a nontrivial rational solution \begin{equation}(X_{0}, Y_{0}, Z_{0}, W_{0}, M_{0}, N_{0}).\end{equation} Therefore, system (3.34) has $(X_{0}, Y_{0}, Z_{0}, W_{0}, M_{0}, N_{0})$ as a common nontrivial rational solution.

From (3.44)-(3.50), it can be verified that there exist nonzero rational numbers $L_{1}$, $L_{2}$, $L_{3}$, $L_{4}$, $L_{5}$, and $L_{6}$ such that:
\begin{equation}
F_{x}(X, Y, Z, W, M, N) = L_{1}(M^{2} - XZ),
\end{equation}
\begin{equation}
F_{y}(X, Y, Z, W, M, N) = L_{2}(M^{2} - XZ),
\end{equation}
\begin{equation}
WM - XN = L_{3}(M^{2} - XZ),
\end{equation}
\begin{equation}
MN - ZW = L_{4}(M^{2} - XZ),
\end{equation}
\begin{equation}
WN - YM = L_{5}(M^{2} - XZ),
\end{equation}
\begin{equation}
N^{2} - YZ = L_{6}(M^{2} - XZ).
\end{equation}

From (3.51) and (3.57), it follows that \begin{equation}N_{0}^{2} - Y_{0}Z_{0} = (\sqrt[2]{L_{6}}M_{0})^{2} - (L_{6}X_{0})Z_{0} \end{equation} and thus \begin{equation}(N_{0} - \sqrt[2]{L_{6}}M_{0})(N_{0} + \sqrt[2]{L_{6}}M_{0}) = (Y_{0} - L_{6}X_{0})Z_{0}. \end{equation}  Hence, \begin{equation}N_{0} = \begin{cases}-\sqrt[2]{L_{6}}M_{0}, or \\ +\sqrt[2]{L_{6}}M_{0} \end{cases}\end{equation} and \begin{equation}Y_{0} = L_{6}X_{0}.\end{equation} It follows from (3.51) and (3.60) that $$\frac{N_{0}}{M_{0}} = \begin{cases}\frac{-\sqrt[2]{L_{6}}M_{0}}{M_{0}}, or \\ \frac{+\sqrt[2]{L_{6}}M_{0}}{M_{0}}\end{cases}$$ implies that $\sqrt[2]{L_{6}}$ is a rational number, i.e., $L_{6}$ is a square of a rational number.

Let us choose $X_{1}$ and $Z_{1}$ to be any two nonzero rational numbers such that $X_{1}Z_{1}$ is a square of a rational number. Then it follows from (3.52)-(3.57) that there exists a nontrivial rational solution $(X_{0}, Y_{0}, Z_{0}, W_{0}, M_{0}, N_{0})$ such that $X_{0} = X_{1}$, $Y_{0} = L_{6}X_{1}$, and $Z_{0} = Z_{1}$. Thus $X_{0}Y_{0}Z_{0} \neq 0$ as required.

Therefore, in either case, cubic form (3.2) has a nontrivial rational solution by Proposition 3.3, Proposition 3.5, Key Proposition, and Remark 3.4.

\end{proof}

\begin{rem}
From the proof of the Key Proposition above, the Key Proposition also holds if system (3.34) is replaced by \begin{equation}\begin{cases} F_{y}(X, Y, Z, W, M, N) = 0 \\ F_{z}(X, Y, Z, W, M, N) = 0 \\ WM - XN = 0 \\ MN - ZW = 0 \\ WN - YM = 0 \\ W^{2} - XY = 0 \\ M^{2} - XZ = 0 \end{cases}\end{equation} or \begin{equation}\begin{cases} F_{x}(X, Y, Z, W, M, N) = 0 \\ F_{z}(X, Y, Z, W, M, N) = 0 \\ WM - XN = 0 \\ MN - ZW = 0 \\ WN - YM = 0 \\ W^{2} - XY = 0 \\ N^{2} - YZ = 0. \end{cases}\end{equation}
\end{rem}

\end{proof}

To continue with the proof of Theorem 2.1, let $E$ be an elliptic curve over $\mathbb{Q}$. Let $$\mathfrak{L}: = \{\mathcal{C}\}$$ be the collection of all equivalent classes of homogeneous spaces of $E$. Thus $E$ is the Jacobian of each of these curves $\mathcal{C}$. It is known that each $\mathcal{C}$ in $\mathfrak{L}$ can be represented as a smooth cubic plane curve $F(x, y, z) = 0$ with rational coefficients.

If none of these cubic forms $F(x, y, z) = 0$ can be written in the form $$F(x, y, z) = A_{1}x^{3} + A_{2}y^{3} + A_{3}z^{3}$$ with $A_{1}$, $A_{2}$, $A_{3}$ rational numbers and $A_{1}A_{2}A_{3} \neq 0$, then $\underline{III}(E)$ must be trivial by the Key Lemma above.

Thus let us suppose that at least one $\mathcal{C}$ in $\mathfrak{L}$ can be represented by a smooth cubic form $F(x, y, z) = 0$ of the form \begin{equation}F(x, y, z) = A_{1}x^{3} + A_{2}y^{3} + A_{3}z^{3}\end{equation} where $A_{1}$, $A_{2}$, and $A_{3}$ are rational numbers with $A_{1}A_{2}A_{3} \neq 0$. By multiplying both sides of (3.64) by an appropriate integer, we may assume that $A_{1}$, $A_{2}$, $A_{3}$ are cube free integers and $GCD(A_{1}, A_{2}, A_{3}) = 1$.

\begin{prop}
Let $\mathcal{C}$ be an element of $\mathfrak{L}$ such that $\mathcal{C}$ can be written in the form (3.64). Then $E$ is birationally equivalent over $\mathbb{Q}$ to the following curve: \begin{equation}Y^{2} = X^{3} - 432(A_{1}A_{2}A_{3})^{2}.\end{equation}
\end{prop}

\begin{proof}
To establish this, we use a construction of Selmer ([19]). We consider the matrices
\begin{equation}
A = \left(
  \begin{array}{ccc}
    0 & 0 & -9A_{3} \\
    27A_{3}\sqrt[3]{A_{1}} & -27A_{3}\sqrt[3]{A_{2}} & 0 \\
    \frac{3\sqrt[3]{A_{1}}}{4} & \frac{3\sqrt[3]{A_{2}}}{4} & 0 \\
  \end{array}
\right)
\end{equation} and

\begin{equation}
A^{-1} = \left(
  \begin{array}{ccc}
    0 & \frac{1}{54A_{3}\sqrt[3]{A_{1}}} & \frac{2}{3\sqrt[3]{A_{1}}} \\
    0 & \frac{-1}{54A_{3}\sqrt[3]{A_{2}}} & \frac{2}{3\sqrt[3]{A_{2}}} \\
    \frac{-1}{9A_{3}} & 0 & 0 \\
  \end{array}
\right)
\end{equation} which  represent the maps \begin{equation}\begin{array}{ccc}
                                                           x = \frac{36A_{3}Z + Y}{54\sqrt[n]{A_{1}}A_{3}}, & y = \frac{36A_{3}Z  - Y}{54\sqrt[n]{A_{2}}A_{3}},& z = -\frac{X}{9A_{3}}
                                                         \end{array}
\end{equation} and \begin{equation} \begin{array}{ccc}
                                      X = -9A_{3}z, & Y = 27A_{3}(\sqrt[3]{A_{1}}x - \sqrt[3]{A_{2}}y), & Z = \frac{3}{4}(\sqrt[3]{A_{1}}x + \sqrt[3]{A_{2}}y).
                                    \end{array}
\end{equation} By an additional change of variable \begin{equation}x = \frac{X}{\sqrt[3]{A_{1}A_{2}}} \end{equation} together with a scaling of $x$ and $Y$, we have \begin{equation}Y^{2} = X^{3} - 432(A_{1}A_{2}A_{3})^{2}.\end{equation} Therefore the curve in (3.71) is the Jacobian of all the curves $\mathcal{C}$ in $\mathfrak{L}$ and thus is birationally equivalent to $E$ over $\mathbb{Q}$.
\end{proof}

It can be verified from (3.71) that there are finitely many distinct inequivalent cubic forms of the form (3.64) having the curve (3.71) as its Jacobian. Therefore, $E$ only has at most finitely many inequivalent homogenous spaces of the form (3.64). As a result, it follows that $\underline{III}(E)$ must have finite order.

\end{proof}

\begin{rem}
It can be seen from the proof of Theorem 2.1 that, for a given elliptic curve $E$ over $\mathbb{Q}$, the order of $\underline{III}(E)$ is in principle computable.
\end{rem}

\begin{proof}{(Proof of Corollary 2.3)}

In [8], T. Dokchitzer and V. Dokchitzer prove that the Parity Conjecture for an elliptic curve $E$ over $\mathbb{Q}$, namely $$r_{anal}(E) \equiv r_{WM}(E) \pmod{2},$$ follows from the finiteness of the Tate-Shafarevich group $\underline{III}(E)$. As a result, Theorem 2.1 implies the Parity Conjecture. Thus the first statement in the strong Birch and Swinnerton Conjecture, $$r_{anal}(E) = r_{WM} (E),$$ is known modulo 2 for any elliptic curve $E$ over $\mathbb{Q}$.

\end{proof}

\begin{proof}{(Proof of Corollary 2.4)}

Up to this point, all the quantities appearing in (1.2) are known to be computable completely except $|\underline{III}(E)|$ and $Reg(E)$. \begin{itemize}
\item $\Omega(E)$ can be computed to any precision using the doubly exponential AGM algorithm.
\item For each $p$, $c_{p}$ can be computed using Tate's algorithm.
\item The order of the group of torsion points $|E(\mathbb{Q})_{tor}|$ is computable.
\item The regulator $Reg(E)$ can be computed to any desired precision if one knows generators for the Mordell group $E(\mathbb{Q})$.
\end{itemize} With regard to the term $|\underline{III}(E)|$, Theorem 2.1 shows that $\underline{III}(E)$ is finite and the proof of Theorem 2.1 shows that $|\underline{III}(E)|$ is computable in principle.
Now let us show that Theorem 2.1 guarantees that the standard process for producing generators of $E(\mathbb{Q})$ reaches its goal in a finite number of steps.

Let $h$ be the canonical height function on the Mordell group $E(\mathbb{Q})$. Let $n >1$ be an integer and $\mathcal{E} = \{e_{j}\}$ be the set of representatives for the cosets of $nE(\mathbb{Q})$ in $E(\mathbb{Q})$. It is known that if there exists a positive integer $H$ such that $h(e_{j}) < H$ for all $j$, then $E(\mathbb{Q})$ is generated by the collection $\mathfrak{E}$ of all points $\overline{e}$ of $E(\mathbb{Q})$ with \begin{equation}h(\overline{e}) < H.\end{equation} It is known that (3.72) implies $\mathfrak{E}$ has finite order. Thus one needs $\mathcal{E}$ to be of finite order. Consider the exact sequence \begin{equation} E(\mathbb{Q}) \overset{[n]}{\longrightarrow} E(\mathbb{Q}) \overset{\delta}{\longrightarrow} Sel^{[n]}(E) \longrightarrow \underline{III}_{n}(E)\end{equation} where \begin{itemize}
\item $Sel^{[n]}(E)$ is the $n$-Selmer group of $E$, which is a computable finite group.
\item $\underline{III}_{n}(E)$ is the subgroup of $\underline{III}(E)$ of elements of order dividing $n$ for which there is not yet an effective method for computing.
\item $[n]$ is the multiplication by $n$ map.
\end{itemize} For each positive integer $l \geq 1$, there is a commutative diagram
\begin{equation}\begin{array}{ccccccc}
E(\mathbb{Q}) & \longrightarrow & Sel^{[n^{l}]}(E) & \longrightarrow & \underline{III}_{n^{l}}(E) & \longrightarrow & 0 \\
\downarrow [1] &  & \downarrow \gamma_{n} &  & \downarrow [n^{l-1}] &  &  \\
E(\mathbb{Q}) & \overset{\delta}{\longrightarrow} & Sel^{[n]}(E) & \longrightarrow & \underline{III}_{n}(E) & \longrightarrow & 0
\end{array}\end{equation} which leads to the following exact sequence \begin{equation}E(\mathbb{Q}) \overset{[n]}{\longrightarrow} E(\mathbb{Q}) \overset{\delta}{\longrightarrow} Sel^{[n]}_{l}(E) \longrightarrow n^{l-1}\underline{III}(E)\end{equation} where: \begin{itemize}
\item $\gamma_{n}$ denotes the $n$th-descent which is computable.
\item $\gamma_{n}(Sel^{[n^{l}]}(E)) = Sel^{[n]}_{l}(E)$.
\end{itemize}

The standard process for producing generators for $E(\mathbb{Q})$ is as follows: Let $L_{i}(E)$ denote the subgroup of $Sel^{[n]}(E)$ generated by all the elements $\delta(e)$ where $e$'s are points on $E(\mathbb{Q})$ such that their $x$ coordinates satisfy $h(x) \leq i$. Next, compute:
\begin{enumerate}
\item $Sel^{[n]}_{1}(E) \supset Sel^{[n]}_{2}(E) \supset \ldots$.
\item $L_{1}(E) \subset L_{2}(E) \subset \ldots$.
\end{enumerate} If $L_{i}(E) = Sel^{[n]}_{l}$ for some $i$ and $l$, then it follows that $$n^{l-1}\underline{III}_{n^{l}}(E) = 0$$ and the points $e$ with $x$ coordinates satisfying $h(x) \leq i$ generate $E(\mathbb{Q})/nE(\mathbb{Q})$. The corresponding generators for $E(\mathbb{Q})$ can thus be found. This process will not stop if $\underline{III}_{n}(E)$ contains an infinitely divisible element. Theorem 2.1 guarantees that such an element does not exist and thus the above process must stop in a finite number of steps, producing generators for $E(\mathbb{Q})$.

Therefore, the term $c_{E}$ in the strong Birch and Swinnerton-Dyer Conjecture can be computed, in principle, for any elliptic curve $E$ over $\mathbb{Q}$.

\end{proof}

\begin{proof}{(Proof of Corollary 2.5)}

This is outlined in [17]. The procedure can be stated briefly as follows: Let $\mathcal{C}$ be a genus 1 curve with rational coefficients and let $Jac(\mathcal{C})$ be its Jacobian. Let $$\Phi: \underline{III}(Jac(\mathcal{C})) \times \underline{III}(Jac(\mathcal{C})) \rightarrow \mathbb{Q}/\mathbb{Z}$$ be the Cassels-Tate pairing. Then:
\begin{itemize}

\item If $\mathcal{C}(\mathbb{Q}_{p}) = \emptyset$ for some finite or infinite prime $p$, then the Mordell-Weil group $\mathcal{C}(\mathbb{Q}) = \emptyset$.
\item If the Mordell Weil group $\mathcal{C}(\mathbb{Q}) \neq \emptyset$, then a rational point can be found by searching.
\item If $\mathcal{C}(\mathbb{Q}_{p}) \neq \emptyset$ for all finite and infinite primes $p$ and the Mordell Weil group $\mathcal{C}(\mathbb{Q}) = \emptyset$, then $\mathcal{C}$ is a nontrivial element of $\underline{III}(Jac(\mathcal{C}))$. This can be proven, given that $\underline{III}(Jac(\mathcal{C}))$ is finite by Theorem 2.1, by finding another element $\mathcal{C}^{\prime}$ of $\underline{III}(Jac(\mathcal{C}))$ such that $$\Phi(\mathcal{C}, \mathcal{C}^{\prime})$$ is nontrivial in $\mathbb{Q}/\mathbb{Z}$.
\end{itemize}
\end{proof}


\end{document}